\documentclass[11pt,leqno]{amsart}
\usepackage{graphicx} 
\usepackage[margin=1.1in]{geometry}
\usepackage{times}
\usepackage{pdfsync}
\usepackage{dsfont}
\usepackage{color} 
\usepackage{amssymb}
\usepackage{bm} 
\usepackage{mathtools}

\usepackage{caption}
\usepackage{microtype}
\usepackage{enumerate}
\usepackage{mathabx}

\newcommand{\R}{\mathbb{R}}

\newcommand{\tacka}{\,\cdot\,}

\DeclareMathOperator{\diam}{diam}

\newtheorem{theorem}{Theorem}[section]
 
\newtheorem{corollary}[theorem]{Corollary}
\newtheorem{definition}[theorem]{Definition}
\newtheorem{example}[theorem]{Example}
\newtheorem{lemma}[theorem]{Lemma}

\theoremstyle{remark}
\newtheorem{remark}[theorem]{Remark}

\definecolor{cherry}{rgb}{0.8,0.1,0.5}

\def\dd{\mathrm{d}}
\def\E{\mathbb{E}}

\usepackage{bbm}
\usepackage{mathrsfs}
\usepackage{galois}
\usepackage{float}
\usepackage{subfig}
\usepackage{amsmath}
\usepackage[shortlabels]{enumitem}

\usepackage{mleftright}
\mleftright
\usepackage{centernot}

\usepackage{bbm}
\title[sliced Wasserstein distance on Hilbert spaces]{Sliced Wasserstein distance between probability measures on infinite dimensional Hilbert spaces}
\author{Ruiyu Han}
\address[Ruiyu Han]{Department of Mathematical Sciences\\Carnegie Mellon University\\Pittsburgh, PA 15213}
\email{ruiyuh@andrew.cmu.edu}
\date{}

\begin{document}


\thanks{This work has been partially supported by the National Science Foundation under grant DMS 2206069.}
\maketitle

\begin{abstract}
The sliced Wasserstein distance has been widely studied as a tool for comparing probability measures on $\mathbb{R}^d$. In this work, we rigorously extend the notion of sliced Wasserstein distance to measures on infinite-dimensional separable Hilbert spaces. We further characterize the induced topology in terms of narrow convergence of measures and provide quantitative results on approximation through empirical measures.
\end{abstract}

\section{Introduction}
 Let  $(\mathcal X,d)$ be a Polish metric space and let $\mathcal P(\mathcal X)$ denotes the set of probability measures on $\mathcal X$. The Wasserstein distance of order $p$ with $p\in [1,\infty)$ is defined on the space of probability measures which have a finite moments of order $p$, denoted as $\mathcal P_p(\mathcal X)$, where
\[
\mathcal P_p(\mathcal X)=\{ \mu\in\mathcal P(\mathcal X)\: : \: \int_{\mathcal{X}}d(x,x_0)^p\dd\mu(x)<\infty\},
\]
where $x_0\in\mathcal{X}$ is arbitrary. The Wasserstein distance $W_p$ between two probability measures $\mu,\nu\in \mathcal P_p(\mathcal X)$  is defined as 
\begin{equation}
W_p(\mu,\nu)=\left(\inf\limits_{\pi\in\Pi(\mu,\nu)}\int_{\mathcal{X} \times \mathcal{X}}d(x,y)^p\pi(\dd x,\dd y)\right )^{\frac1p},
\end{equation}
where $\Pi(\mu,\nu)$ is the set of transport plans, i.e.,
\[
\Pi(\mu,\nu)=\{\gamma\in\mathcal P(\mathcal{X}\times \mathcal{X})\: :\: \forall A, B\text{\:Borel\:}, \gamma(A\times B)=\mu(A)\nu(B)\}.
\]


The Wasserstein distance rises as a fundamental metric for quantifying the dissimilarity between probability distributions, boasting a myriad of applications within the realms of statistics and machine learning. However, it suffers from the curse of dimension limiting its application to large-scale data analysis,  under which the empirical measure shows slow convergence to the true distribution as the dimension becomes large~\cite{weed2017sharp,fournier2013rate}. In order to alleviate the computational load, sliced Wasserstein distance emerged~\cite{bonneel2014sliced,Bonnotte_SW} and received a surge of interest since the convergence of empirical measure under sliced Wasserstein distance is independent of the dimension. Many variants of sliced Wasserstein distance on $\mathbb R^d$ have been explored, such as the generalized sliced Wasserstein distance~\cite{kolouri2019generalized}, energy-based sliced Wasserstein distance, and hierarchical
sliced Wasserstein distance~\cite{nguyen2023hierarchical}. In particular, the sliced Wasserstein distance has been the focus of much research in recent years, to name a few~\cite{park2024geometryanalyticpropertiessliced, cozzi2024longtimeasymptoticsslicedwassersteinflow, bonet2024slicedwassersteindistancesflowscartanhadamard}.

A natural question arises that if it is possible to define a extended notion of sliced Wasserstein distance for measures on infinite dimensional spaces, as the Wasserstein distance exists for such measures. In particular, the convergence rate of the expectation of the sliced Wasserstein distance between the true and empirical measure is of interested. If the convergence rate is independent of dimension, then it implies that the sliced Wasserstein distance does not suffer from the curse of dimensionality. The work~\cite{rustamov2023intrinsic} defines the sliced Wasserstein distance on compact manifolds and provides real data examples. In this paper, we establish the notion of sliced Wasserstein distance between measures on an infinite dimensional separable Hilbert space in a more theoretical view, which also allows for noncompact domains. We show that empirical measures approximate the true distribution in sliced Wasserstein distance at  parametric rates. We also show that the sliced Wasserstein distance characterize the narrow convergence of measures.

The definition of sliced Wasserstein distance~\eqref{def: sw_infinite} for measures on infinite dimensional spaces resembles that of measures on $\mathbb R^d$, where the main task is to make the surface integral well-defined. The newly-defined sliced Wasserstein distance indeed depicts the narrow convergence of measure similarly as Wasserstein distance does~\cite{santambrogio,villani}, but it turns out to require stronger conditions than that in $\mathbb R^d$ to infer the asymptotic behaviour of measures, see Theorem~\ref{sw: weak topo_w} below. In particular, the requirement of further assumptions origins from the loss of compactness of the unit sphere. Meanwhile, the approximation via empirical measures  survives from the curse of dimension, see Theorem~\ref{thm: empirical}. It shares the same behaviour as the $p$-sliced Wasserstein distance on between probability measures on $\mathbb R^d$ ~\cite{Manole_2022}. Compared to the results for Wasserstein distances~\cite{fournier2013rate,10.3150/19-BEJ1151}, the sliced Wasserstein distance reveals its computational efficiency, see Subsection~\ref{subsec: comparison_Wasserstein} for further details.

Whether the notion of sliced Wasserstein distance has a parallel definition in general infinite dimensional Banach space is unknown, but we point out that, within the scope of this paper, the requirement of Hilbert space is crucial for the sliced Wasserstein distance to be well-defined. In particular, inner product and decomposition theorem allow the projection to resemble that in the Euclidean space. We also point out that for measures on $\mathbb R^d$, there is another equivalent definition of sliced Wasserstein distance using Radon transform~\cite{bonneel2014sliced,nguyen2023hierarchical}. Radon transform does have several extensions in infinite dimensions~\cite{becnel,bogachev_radon,Mihai}, but they either appear hard to tackle~\cite{bogachev_radon} or only apply to $L^2$ functions on a certain probability space~\cite{becnel,Mihai}. Further exploration in this direction is welcomed. 

The structure of the rest of paper is outlined as follows. In Section~\ref{section: def of sw} we will provide a rigorous definition of sliced Wasserstein distance between measures on an infinite dimensional separable Hilbert space. Section~\ref{section: narrow_conv_sw} is devoted to characterize the narrow convergence of measures via the newly-defined sliced Wasserstein distance. Finally, in Section~\ref{section: empirical} we study the convergence rate of empirical measure, which is consistent with those results in finite dimensions~\cite{Manole_2022}.

\subsection{Notation}
In the following context, let the order $p\in [1,\infty)$. $X$ is an infinite dimensional separable Hilbert space, where the norm denoted by $\|\cdot\|$ is induced by the inner product $\langle \cdot \rangle$. The space of probability measures with finite $p$-moments can be simplified as
\[
\mathcal P_p(X)=\{ \mu\in\mathcal P(\mathcal X)\: : \: \int_{\mathcal{X}}\|x\|^p\dd\mu(x)<\infty\}.
\]
For $\mu\in\mathcal P_p(X)$, define $M_p(\mu)=\int_X\|x\|^p\dd\mu(x)$.

\textit{Pushforward measure}: given a probability measure $\mu$ on $X$ and a unit vector $\theta\in X$, we define the pushforward measure $\hat{\mu}_\theta$ via $$\hat{\mu}_\theta(A)=\mu(\{x\in X: \langle x, \theta\rangle \in A\}),\qquad  A\subseteq \R \text{ Borel}.$$


\textit{Empirical measure: } Let $X_1, \dots, X_n$ be an i.i.d.\ sample from $\mu$, the empirical measure $\mu^n$ of $\mu$ is defined as
\[
\mu^n:= \frac{1}{n} \sum_{j=1}^n \delta_{X_j}.
\]
\section{Sliced 
Wasserstein Distance on $\mathcal P_p(X)$}
\label{section: def of sw}

This section is devoted to establishing a well-defined notion of sliced Wasserstein distance between measures in $\mathcal P_p(X)$. Before that, let's recall the definition of sliced Wasserstein distance on $\mathcal P_p(\mathbb R^d).$
For $\mu,\nu\in\mathcal P_p(\mathbb R^d)$, the sliced Wasserstein distance of order $p\geq 1$, denoted as $SW_p$, is defined as follows: 
\begin{equation}\label{sw: finite}
SW^p_p(\mu,\nu)=\frac{1}{\mathcal H(\mathbb S^{d-1})}\int_{\mathbb S^{d-1}}W_p^p(\hat\mu_{\theta},\hat\nu_{\theta})\dd \mathcal H^{d-1}(\theta),
\end{equation}
where $\mathcal H^{d-1}$ denotes the $d-1$ dimensional Hausdorff measure and $\mathcal H^{d-1}(\mathbb S^{d-1})$ denotes the surface area of the $d-1$ dimensional unit sphere. 
For two measures $\mu,\nu\in\mathcal P(X)$, we aims to construct an analogy as~\eqref{sw: finite}.

In infinite dimensional space there is no longer a compact unit sphere that is $1$ dimension less than the space dimension. To have an analogy as $\mathbb S^{d-1}$ in $\mathbb R^d$, we take the candidate $S:=\{x\in X| \|x\|=1\}$ which consists of unit vectors in every direction. Our goal is to make the following formal integral well-defined
\begin{equation}\label{formal_sw}
\frac{1}{\gamma_S(S)}\int_{\|\theta\|=1}W_p^p(\hat\mu_{\theta},\hat\nu_{\theta})\gamma_S(\dd\theta),
\end{equation}
where $\gamma_S$ is some finite Borel measure on $S$. 

\subsection{Surface measure on unit sphere} \label{subsection: surface measure}
Surface measure of infinite dimensional spaces is a topic of its own interest. One eligible path to find a finite and strictly positive Borel measure defined on unit sphere $S$ is that we first define a strictly positive probability measure $\gamma$ on the whole space $X$ and then take $\gamma_S$ to be the surface measure associated to $\gamma$. The existence of such surface measure is a nontrivial task. Initially it is only defined for sufficiently regular surface using tools from Malliavin calculus~\cite{bogachev_gaussian}; later the restrictions are reduced in~\cite{daprato2016malliavin,daprato2014surface,Da_Prato_surface}. 



We refer to~\cite[Section 3]{Da_Prato_surface} for solid examples of measure $\gamma$.  Then we define the surface measure $\gamma_S$ concentrated on $S$ associated to $\gamma$. Let $S^{\epsilon}:=\{x\in X\:|\: 1-\epsilon\leq\|x\|\leq 1+\epsilon\}$ and let $f: X\to\mathbb R$  be a Borel function defined. Set
\begin{equation}
\int f(x)\dd \gamma_S=\lim\limits_{\epsilon\to 0}\frac{1}{2\epsilon}\int_{S^{\epsilon}}f(x)\,\gamma(\dd x).
\end{equation}

According to~\cite[Theorem 2.11, Proposition 3.5, Example 3.8]{Da_Prato_surface}, there exists a unique Borel measure $\gamma_S$ whose support is included in $S$, such that for $\varphi: X\to\mathbb R$ which is uniformly continuous and bounded,
\begin{equation}\label{eq: F_varphi}
\int_{X}\varphi(x)\gamma_S(\dd x)=(F_{\varphi}(r))'|_{r=1}, \quad F_{\varphi}(r):=\int_{\|x\|^2\leq 1}\varphi(x)\gamma(\dd x).
\end{equation}

Taking $\varphi\equiv 1$ yields  $\int_{\|\theta\|=1}1\gamma_S(\dd\theta)<\infty$. It can be easily checked that if $\gamma$ is strictly positive, then $\gamma_S$ is strictly positive. In particular, we can pick $\gamma$ a non degenerate centered Gaussian measure on $X$. Recall the definition of Gaussian measure in infinite dimensions~\cite{eldredge2016analysis}:
\begin{definition}[Infinite-dimensional Gaussian measures]
    Let $W$ be a topological vector space and $\mu$ a Borel probability on $W$. $\mu$ is Gaussian if and only if, for each continuous linear functional $f$ on $W^*$, the pushforward $\mu\comp f^{-1}$ is a Gaussian measure on $\mathbb R$.
\end{definition}
Since $X$ is a separable Hilbert space, there is a more explicit description of a non degenerated centered Gaussian measure $\gamma$ on $X$. By Karhunen-Lo\`eve expansion~\cite{Adler},  $\gamma=\mathcal L(\sum\limits_{i=1}^{\infty}\lambda_i\xi_i e_i)$. $\mathcal L$ denotes the law, $\{\xi_i\}_{i\in\mathbb N}$ are i.i.d. standard Gaussian. The eigenvalues $\{\lambda_i\}_{i\in\mathbb N}$ satisfy $\lambda_i\neq 0$ and $\sum\limits_{i=1}^{\infty}\lambda_i^2<\infty$.

\subsection{Wasserstein distance between projected measures}
\label{subsection: uniform_continuity}
In this subsection we will prove the uniform continuity and bound of $W_p^p(\hat\mu_{\theta},\hat\nu_{\theta})$ given that the measures have appropriate moments, which ensures that $W_p^p(\hat\mu_{\theta}, \hat\nu_{\theta})$ is integrable with respected to $\gamma_S$.

We first demonstrate that the quantity $W_p^p(\hat\mu_{\theta},\hat\nu_{\theta})$ is well-defined. Given a unit 
vector $\theta\in X$, the pushforward measure $\hat\mu_{\theta}:=P_{\theta\#}\mu$ is a probability measure on $\mathbb R$ and $\hat\mu_{\theta}\in\mathcal P_p(\mathbb R)$; indeed, by change of variables,
\begin{equation}\label{est: p moment}
\int_{\mathbb R}|y|^p\,\dd \hat\mu_{\theta}(y) = \int_{X}\|\tilde P_{\theta}(x)\|^p\,\dd \mu(x)\leq \int_{X}\|x\|^p\,\dd \mu(x)<\infty.
\end{equation}
Therefore, for $\mu,\nu\in\mathcal P_p(X)$ and a unit 
vector $\theta\in X$, the Wasserstein distance $W_p(\hat\mu_{\theta},\hat\nu_{\theta})$ is well-defined.


Now we are ready to check $W^p_p(\hat\mu_{\theta},\hat\nu_{\theta})$ is uniformly continuous and bounded on $S$. Observe that the function $W_p(\hat\mu_{\theta},\hat\nu_{\theta})$ is Lipschitz on $S$. 

\begin{lemma} \label{est: pushforward_lip} Given that $\mu,\nu\in \mathcal P_p(X)$,
$W_p(\hat\mu_{\theta},\hat\nu_{\theta})$ is Lipschitz on $S$ with Lipschitz constant $(M_p(\mu))^{\frac1p}+(M_p(\nu))^{\frac1p}$.
\end{lemma}

\begin{proof}
  Let $\theta,\gamma \in S$. Triangle inequality gives that
\begin{equation*}
|W_p(\hat\mu_{\theta},\hat\nu_{\theta})-W_p(\hat\mu_{\gamma},\hat\nu_{\gamma})|\leq W_p(\hat\mu_{\gamma},\hat\mu_{\theta})+W_p(\hat\nu_{\gamma},\hat\nu_{\theta}).
\end{equation*}
Notice that $\pi_{\mu}:=(P_{\theta}\times P_{\gamma})\#\mu$ is a transport plan between $\hat\mu_{\theta}$ and $\hat\mu_{\gamma}$. Then
\begin{align*}
W_p^p(\hat\mu_{\theta},\hat\mu_{\gamma})\leq \int_{\mathbb R^2}|y-z|^p\dd \pi_{\mu}(y,z)&=\int_{X}|P_{\theta}(x)-P_{\gamma}(x)|^p\dd\mu(x)\\
&= \int_X |\langle \theta-\gamma,x\rangle|^p\dd \mu(x)\leq \|\theta-\gamma\|^p \int_X\|x\|^p\dd\mu(x),
\end{align*}
where the last inequality we use Cauchy-Schwarz inequality. The above argument implies that
\[
W_p(\hat\mu_{\theta},\hat\mu_{\gamma})\leq \|\theta-\gamma\|(M_p(\mu))^{\frac1p}.
\]
 Analogously, $W_p(\hat\nu_{\theta},\hat\nu_{\gamma})\leq \|\theta-\gamma\|(M_p(\nu))^{\frac1p}$. Therefore,
\begin{equation}\label{eq: est_lip_pushforward}
|W_p(\hat\mu_{\theta},\hat\nu_\theta)-W_p(\hat\mu_{\gamma},\hat\nu_\gamma)|\leq \|\theta-\gamma\|((M_p(\mu))^{\frac1p}+(M_p(\nu))^{\frac1p}).\qedhere
\end{equation}
\end{proof}
Equipped with Lemma~\ref{est: pushforward_lip}, we conclude this subsection with the following theorem:
\begin{theorem}\label{thm: uni_cont_bound}
Given that $\mu,\nu\in \mathcal P_p(X)$, $W^p_p(\hat\mu_{\theta},\hat\nu_{\theta})$ is bounded on S, in particular 
\begin{equation}
\forall\theta\in S,\quad W^p_p(\hat\mu_{\theta},\hat\nu_{\theta})\leq 2^p\left(M_p(\mu)+M_p(\nu)\right).
\end{equation}

Meanwhile, $W^p_p(\hat\mu_{\theta},\hat\nu_{\theta})$ is Lipschitz on $S$ with Lipschitz constant \begin{equation}p 2^{p-1}\max\{M_p(\mu),M_p(\nu)\}^{\frac{p-1}{p}}\left((M_p(\mu))^{\frac1p}+(M_p(\nu))^{\frac1p}\right).
\end{equation}

\end{theorem}
\begin{proof}[Proof of Theorem~\ref{thm: uni_cont_bound}]
\emph{Bound:} 
 Let $\pi_{\theta}\in\Pi(\hat\mu_{\theta},\hat\nu_{\theta})$. Recall that $\Pi(\hat\mu_{\theta},\hat\nu_{\theta})$ consists of probability measures on $\mathbb R\times\mathbb R$ with marginals $\hat\mu_{\theta}$ and $\hat\nu_{\theta}$, respectively. For any $\theta\in S$,
\begin{equation*}
W_p^p(\hat\mu_{\theta},\hat\nu_{\theta})\leq\int_{\mathbb R^2}|x-y|^p\pi_{\theta}(\dd x,\dd y)\leq 2^p\int_{\mathbb R^2}(|x|^p+|y|^p)\pi_{\theta}(\dd x,\dd y)
\overset{\eqref{est: p moment}}{\leq}2^p(M_p(\mu)+M_p(\nu)).
\end{equation*}
\emph{Uniform Continuity:} For $\theta,\gamma\in S$,
\begin{align*}
\MoveEqLeft|W^p_p(\hat\mu_{\theta},\hat\nu_{\theta})-W^p_p(\hat\mu_{\gamma},\hat\nu_{\gamma})|
\leq p\max\{W^{p-1}_p(\hat\mu_{\theta},\hat\nu_{\theta}),W^{p-1}_p(\hat\mu_{\gamma},\hat\nu_{\gamma})\}\tacka |W_p(\hat\mu_{\theta},\hat\nu_{\theta})-W_p(\hat\mu_{\gamma},\hat\nu_{\gamma})|\\
&\leq  p 2^{p-1}\max\{M_p(\mu),M_p(\nu)\}^{\frac{p-1}{p}}\tacka|W_p(\hat\mu_{\theta},\hat\nu_{\theta})-W_p(\hat\mu_{\gamma},\hat\nu_{\gamma})|\\
&\overset{\eqref{eq: est_lip_pushforward}}{\leq}p 2^{p-1}\max\{M_p(\mu),M_p(\nu)\}^{\frac{p-1}{p}}\left((M_p(\mu))^{\frac1p}+(M_p(\nu))^{\frac1p}\right)\|\theta-\gamma\|,
\end{align*}
where the first inequality we use $|a^p-b^p|\leq  p \max\{a,b\}^{p-1}|a-b|$ for $a,b\in\mathbb R$, $a,b\geq 0$ and $p\in [1,\infty)$.

\end{proof}
\subsection{Sliced Wasserstein distance on $\mathcal P_p(X)$}

Now we are well-equipped to make the formal expression~\eqref{formal_sw} rigorous.
\begin{definition}
Given $\gamma_S\in\mathcal P(S)$ be strictly positive Borel measure defined on $S$ such that $\gamma_S(S)=\int_{\|\theta\|=1}1\gamma_S(\dd\theta)<\infty$. Let $\mu,\nu\in\mathcal P_p(X)$, the $p$-sliced Wasserstein distance (with respect to $\gamma_S$) is defined as
\begin{equation}\label{def: sw_infinite}
SW^{\gamma}_p(\mu,\nu)=\left (\frac{1}{\gamma_S(S)}\int_{\|\theta\|=1}W^p_p(\hat\mu_{\theta},\hat\nu_{\theta})\gamma_S(\dd \theta)\right )^{\frac1p}.
\end{equation}
\end{definition}

Recall that in $\mathbb R^d$, the Borel measure on $\mathbb{S}^{d-1}$ is often taken as $\mathcal{H}^{d-1}$. In the infinite dimensional case, we can also choose some special measure. In particular, we can take $\gamma_S$ the surface measure on $S$ associated to a non degenerate centered Gaussian measure on $X$. 
\begin{definition}[Gaussian as reference measure]
Given $\gamma\in\mathcal P(X)$ be a non degenerate centered Gaussian measure on $X$. Let $\mu,\nu\in\mathcal P_p(X)$, the $p$-sliced Wasserstein distance (with respect to $\gamma$) is defined as
\begin{equation}\label{def: sw_infinite_surface}
SW^{\gamma}_p(\mu,\nu)=\left (\frac{1}{\gamma_S(S)}\int_{\|\theta\|=1}W^p_p(\hat\mu_{\theta},\hat\nu_{\theta})\gamma_S(\dd \theta)\right )^{\frac1p},
\end{equation}
where $\gamma_S$ is the surface measure on $S$ associated to $\gamma$, and $\gamma_S(S)=\int_{\|\theta\|=1}1\gamma_S(\dd\theta)$.
\end{definition}

The quantity~\eqref{def: sw_infinite_surface} is well-defined owing to the discussion in Subsection~\ref{subsection: surface measure} and~\ref{subsection: uniform_continuity}. Next, we show that~\eqref{def: sw_infinite} is indeed a distance.
\begin{theorem}\label{prop: sw_distance}
Given $\gamma_S\in\mathcal P(S)$ be a finite, strictly positive Borel measure defined on $S$. For $p\in [1,\infty)$, the $SW_p^{\gamma}$ defined as~\eqref{def: sw_infinite} is a distance.
\end{theorem}

\begin{proof}
The symmetry is obvious. The triangle inequality follows from the triangle inequality of $W_p$. Let $\mu,\nu\in\mathcal P_p(X)$ with compact support. If $\mu=\nu$, then $SW_p^{\gamma}(\mu,\nu)=0$. It remains to check that if $SW_p^{\gamma}(\mu,\nu)=0$ implies that $\mu=\nu$.

Notice that since $W_p(\hat\mu_{\theta},\hat\nu_{\theta})$ is nonnegative and uniformly continuous,
  $SW_{p}^{\gamma}(\mu,\nu)=0$ implies that
$W_p(\hat\mu_{\theta},\hat\nu_{\theta})\equiv 0$ for $\|\theta\|=1$. Thus for every $\|\theta\|=1$, $\hat\mu_{\theta}=\hat\nu_{\theta}$ in distribution. Pick an arbitrary $f\in X^*=X$, 
\begin{align*}
\int_{X}e^{if(x)}\mu(\dd x)&=\int_{X}e^{i\|f\|\langle x, \frac{f}{\|f\|}\rangle}\mu(\dd x)=\int_{\mathbb R}\exp(i\|f\|y)\hat\mu_{\frac{f}{\|f\|}}(\dd y)\\
&=\int_{\mathbb
R}\exp(i\|f\|y)\hat\nu_{\frac{f}{\|f\|}}(\dd y)=\int_X e^{if(x)}\nu(\dd x).
\end{align*}
By the injectivity of characteristic functions, we obtain $\mu=\nu$.
\end{proof}

\section{Narrow convergence of measures in $\mathcal P_p(X)$}
\label{section: narrow_conv_sw}
With the definition in hand, we are at the position to the investigate some properties of the sliced Wasserstein distance. It is natural to ask that if the sliced Wasserstein distance~\eqref{def: sw_infinite} can charaterize the narrow convergence of measures on $\mathcal P_p(X)$ since it is well-known that the Wasserstein distance describes the narrow convergece of probability measures~\cite{villani} and  so does sliced Wasserstein distance on $\mathcal P_p(\mathbb R^d)$~\cite{bayraktar2021strong}. In this section we establish the connection between narrow convergence of measures and the quantity of sliced Wasserstein distance.

To begin with, we recall the definition of narrow convergence, although it will not be directly used in the argument below~\cite{AmbrosioLuigi2008GFIM}. Note that in some context it is called ``weak convergence"~\cite{bogachev_weak_convergence_measure}.
\begin{definition}[\cite{AmbrosioLuigi2008GFIM}]
We say that a sequence $\{\mu_n\}\subset \mathcal P(X)$ is narrowly convergent to $\mu\in\mathcal P(X)$ as $n\to\infty$ if
\[
\lim\limits_{n\to\infty}\int_{X}f(x)\mu_n(\dd x)=\int_{X}f(x)\mu(\dd x)
\]
for every $f$ which is a continuous and bounded real function on $X$.
\end{definition}

We will first show that for a narrow convergent sequence with converging $p$-moments, the sliced Wasserstein distance goes to zero.  The inequality between sliced Wasserstein and Wasserstein distance still holds as that in $\mathcal P_p(\mathbb R^d)$~\cite{Bonnotte_SW,bayraktar2021strong}.
\begin{lemma}\label{est: SW<W}
If $\mu,\nu\in\mathcal P_p(X)$, then $SW_{p}^{\gamma}(\mu,\nu)\leq W_p(\mu,\nu)$.
\end{lemma}
\begin{proof}
 There exists an optimal transport plan $\pi$ between $\mu$ and $\nu$ under Wasserstein distance (see Theorem 1.7 in Chapter 1 of~\cite{santambrogio}). Then $(P_{\theta}\times P_{\theta})\# \pi$ is a transport plan between $\hat \mu_{\theta}$ and $\hat \nu_{\theta}$. So
\[
W^p_p(\hat \mu_{\theta},\hat \nu_{\theta})\leq \int_{X^2} |\langle \theta,x\rangle-\langle \theta,y\rangle|^p\dd\pi( x,y).
\]
By Cauchy-Schwarz,
\begin{align*}
SW_{p}^{\gamma}(\mu,\nu)&\leq \left (\frac{1}{\gamma_S(S)}\int_{\|\theta\|=1}\int_{X^2} |\langle \theta,x\rangle-\langle \theta,y\rangle|^p\dd\pi( x,y)\gamma_S(\dd \theta)\right )^{\frac1p}\\
&\leq \left (\frac{1}{\gamma_S(S)}\int_{\|\theta\|=1}\int_{X^2} \|x-y\|^p\|\theta\|^p\dd\pi( x,y)\gamma_S(\dd \theta)\right )^{\frac1p}=W_p(\mu,\nu).
\qedhere
\end{align*}
\end{proof}

Lemma~\ref{est: SW<W} directly gives the following theorem.
\begin{theorem} \label{thm: sw_to_0}
    If $\mu^n,\mu\in\mathcal P_p(X)$, $\mu^n$ converges to $\mu$ narrowly and $\lim\limits_{n\to\infty}\int_{X}\|x\|^p\,\mu^n(\dd x)=\int_{X}\|x\|^p\,\mu(\dd x)$, then $SW^{\gamma}_p(\mu^n,\mu)\to 0$ as $n\to\infty$.
\end{theorem}
\begin{proof}
By Definition 6.8 and Theorem 6.9 in~\cite{villani}, $W_p(\mu^n,\mu)\to 0$. By Lemma~\ref{est: SW<W}, $SW^{\gamma}_p(\mu^n,\mu)\to 0$.
\end{proof}
Now we turn to characterize narrow convergence of measures by the sliced Wasserstein distance. Unlike the finite dimensional case, to have weak convergence, besides the condition that the sliced Wasserstein distance goes to zero, we further require the uniform bound of the $p$-moments. This condition cannot be removed, see Example~\ref{eg: counterexample_narrow_convergence} below for a counterexample.
\begin{theorem}\label{sw: weak topo_w}
If $\mu^n,\mu\in\mathcal P_p(X)$ satisfies $\lim\limits_{n\to\infty}SW_p^{\gamma}(\mu^n,\mu)=0$ and $\sup\limits_{n\geq 1}M_p(\mu_n):=C<\infty$, then $\mu^n$ converges to $\mu$ narrowly.
\end{theorem}

\begin{proof}[Proof of Theorem~\ref{sw: weak topo_w}]
We will prove that every subsequence $\{\mu^{n_k}\}_{k\in\mathbb N}$ admits a further subsequence that converges to $\mu$ narrowly. For the simplicity of notation, denote this subsequence as $\{\mu^n\}_{n\in\mathbb N}$, then we still have $\lim\limits_{n\to\infty}SW_p^{\gamma}(\mu^n,\mu)=0$ and $\sup\limits_{n\geq 1}M_p(\mu_n):=C<\infty$.

As
\[
\lim\limits_{n\to\infty}\int_{\|\theta\|=1}W_p^p(\hat\mu_{\theta}^n, \hat\mu_{\theta})\gamma_S(\dd\theta)=\gamma_S(S)\lim\limits_{n\to\infty}\left (SW^{\gamma}_p(\hat\mu_{\theta}^n, \hat\mu_{\theta})\right )^p=0,
\]
then up to a subsequence $\{n_k\}_{k\in\mathbb N}$, the functions $\theta\mapsto W_p^p(\hat\mu_{\theta}^{n_k}, \hat\mu_{\theta})\in\mathbb R^+$ converges to zero for $\gamma_S$ almost every $\theta\in S$, as $k\to\infty$. 

Meanwhile, given that $\sup\limits_{n\geq 1}M_p(\mu_n):=C<\infty$ and $\mu\in \mathcal P_p(X)$, Proposition~\ref{est: pushforward_lip} implies that for $n\geq 1$ the functions $\theta\mapsto W_p^p(\hat\mu^n_{\theta},\hat\mu_{\theta})$ share the same Lipschitz constant
\[
p 2^{p-1}\max\{(M_p(\mu), C\}^{\frac{p-1}{p}}\tacka(M_p(\mu)^{\frac1p}+C^{\frac1p}).
\]
This implies that as $k\to\infty$, the functions $ W_p(\hat\mu_{\theta}^{n_k},\hat\mu_{\theta})\to 0$ for every $\theta\in S$  since $\gamma$ is nondegenerate Gaussian. It follows that for every $\theta\in S$, $\hat\mu^{n_k}_{\theta}$ converges to $\hat\mu_{\theta}$ narrowly. Now for every $f\in X^*=X$,
\begin{equation}
\begin{split}
&\lim\limits_{k\to\infty}\int_{X}\exp(i\langle f,x \rangle) \mu^{n_k}(\dd x)=\lim\limits_{k\to\infty}\int_{X}\exp\left(i\|f\|\langle \frac{f}{\|f\|},x \rangle\right) \mu^{n_k}(\dd x)\\
=&\lim\limits_{k\to\infty}\int_{\mathbb R}\exp(i\|f\|y)\hat\mu_{\frac{f}{\|f\|}}^{n_k}(\dd y)=\int_{\mathbb R}\exp(i\|f\|y)\hat\mu_{\frac{f}{\|f\|}}(\dd y)=\int_{X}\exp(i\langle f,x \rangle) \mu(\dd x).
\end{split}
\end{equation}
By Proposition 4.6.9 of~\cite{bogachev_weak_convergence_measure}, we obtain that $\mu^{n_k}$ converges to $\mu$ narrowly.
\end{proof}

In particular, when the domain $X$ is bounded, the narrow convergence and a decaying sliced Wasserstein distance are equivalent.
\begin{corollary}
    If $\diam(X)<\infty$, then for $\mu_n,\mu\in\mathcal P_p(X), n\in\mathbb N$, $SW^{\gamma}_{p}(\mu_n,\mu)\to 0$ if and only if $\mu_n$ converges to $\mu$ narrowly.
\end{corollary}
\begin{proof}
The condition $\sup\limits_{n\geq 1}M_p(\mu_n)<\infty$ automatically holds, then Theorem~\ref{sw: weak topo_w} gives that $SW^{\gamma}_{p}(\mu_n,\mu)\to 0$ implies $\mu_n$ converges to $\mu$ narrowly. For the other direction, notice that $x\mapsto \|x\|^p$ is now a bounded continuous function, thus Theorem~\ref{thm: sw_to_0} applies.
\end{proof}

\begin{remark}\label{remark: counterexample_narrow_convergence}
If $X=\mathbb R^d$, then the condition $\sup\limits_{n\geq 1}M_p(\mu^n)<\infty$ can be derived from $SW_p^{\gamma}(\mu^n,\mu)\to 0$, see proof of Theorem 2.1 in~\cite{bayraktar2021strong}. However, we emphasize that, for measures on infinite dimensional space we can no longer obtain $\sup\limits_{n\geq 1}M_p(\mu^n)<\infty$ by the convergence in sliced Wasserstein distance. Consider the following example. 
\end{remark}

\begin{example}\label{eg: counterexample_narrow_convergence}
Let $p=2$ and let $\mu:=\delta_0$ and $\mu^n:=\delta_{n^{\frac13}e_n}$, where $\{e_k\}_{k\in\mathbb N}$ is the orthonormal basis of $X$. Recall that for $\theta\in X$, $\sum\limits_{i=1}^{\infty}|\langle \theta,e_n\rangle|^2=\|\theta\|^2$.  Monotone Convergence Theorem gives that
\[
\lim\limits_{n\to\infty}\sum\limits_{i=1}^n\frac{1}{\gamma_S}\int_{S}|\langle \theta, e_n\rangle|^2\,\gamma_S(\dd\theta)=\frac{1}{\gamma_S}\int_S\sum\limits_{i=1}^{\infty}|\langle \theta,e_n\rangle|^2\,\gamma_S(\dd\theta)=\frac{1}{\gamma_S}\int_S 1\,\gamma_S(\dd\theta)=1,
\]
which implies that $\lim\limits_{n\to\infty}\frac{1}{\gamma_S}\int_S |\langle \theta, e_n\rangle|^2\gamma_S(\dd\theta)=0$. Moreover, we know that \[
\frac{1}{\gamma_S}\int_S |\langle \theta, e_n\rangle|^2\gamma_S(\dd\theta)=o(n^{-1}).\]

On the other hand, for every $\theta\in S$, $W_2^2(\hat\mu^n_{\theta},\hat\mu_{\theta})=n^{\frac23}|\langle \theta, e_n\rangle|^2$. Then we obtain that 
\[
SW_{p}^{\gamma}(\mu^n,\mu)=\frac{1}{\gamma_S}\int_S n^{\frac23}|\langle \theta, e_n\rangle|^2\gamma_S(\dd\theta)=o(n^{\frac23-1})\to 0,\quad n\to\infty.
\]
Meanwhile it is obvious that $M_2(\mu^n)=n^{2/3}$, $\sup\limits_{n\geq 1}M_p(\mu)=\infty$. The $2$-moments are not uniformly bounded. Furthermore, $\mu^n$ do not converge to $\mu$ narrowly. 
\end{example}
\section{Approximation via empirical measures} \label{section: empirical}

The estimate of the distance between empirical measures and its true distribution is a prevailing problem. In this section we investigate the convergence rate of empirical measures on infinite dimensional Hilbert space under the sliced Wasserstein distance~\eqref{def: sw_infinite}; in particular, we have the below theorem: 
\begin{theorem} \label{thm: empirical}
If $\mu\in\mathcal P_s(X)$ for $s>2p$, $\mu^n:=\frac1n\sum\limits_{k=1}^n\delta_{X_k}$ with $X_1,...,X_n$ a sample drawn from $\mu$, then
\begin{equation}\label{conv_rate: empirical}
\E SW^{\gamma}_p (\mu^n,\mu)\leq C n^{-\frac{1}{2p}},
\end{equation}
where the constant $C$ is determined by $p,s$ and $M_s(\mu)$.
\end{theorem}

Theorem~\ref{thm: empirical} immediately gives a concentration inequality of $SW$ using Markov inequality. That is, for any $t>0$,
    \begin{equation}
        \mathbb{P}(SW^{\gamma}_p (\mu^n,\mu)\geq t)\leq C t^{-1}n^{-\frac{1}{2p}}.
    \end{equation}
The proof of Theorem~\ref{thm: empirical} relies on the following result on estimates of one dimensional empirical measure~\cite{bobkov_sergey}.
\begin{theorem}[Theorem 7.16 of~\cite{bobkov_sergey}] \label{thm: empirical_in_1d}Let $X_1,...,X_n$ be an sample drawn from a Borel probability measure $\mu$ on $\mathbb R$ with distribution functions
$F$. Let $\mu^n:=\frac{1}{n}\sum\limits_{k=1}^n\delta_{X_k}$ be the empirical measure. For all $p\geq 1$,
\begin{equation}\label{est: w_p_empirical_1d}
\E W_p^p(\mu^n,\mu)\leq \frac{p2^{p-1}}{\sqrt{n}}\int_{-\infty}^{\infty}|x|^{p-1}\sqrt{F(x)(1-F(x))}\dd x.
\end{equation}
\end{theorem}


Then we bound the right hand side of~\eqref{est: w_p_empirical_1d} by a simple observation. Let $s\geq 1$. Let $\xi$ be a random variable on $\mathbb R$ with distribution function $F$. Assume further $\E|\xi|^s<\infty$. By Chebyshev's inequality, for $x\geq 0$,
\begin{equation*}
\begin{split}
&\left (F(x)(1-F(x)) \right)\tacka (1+|x|^{s})=\left (F(x)(1-F(x)) \right)+ \left (F(x)(1-F(x)) \right) |x|^{s}\\
\leq & 1+ (1-F(x))|x|^{s}\leq 1+\E |\xi|^{s},
\end{split}
\end{equation*}
which implies that for every $x\geq 0$, $F(x)(1-F(x))\leq \frac{1+\E |\xi|^{s}}{1+|x|^{s}}$. The same inequality holds for $x\leq 0$. Thus we have
\begin{equation}\label{est: F(1-F)}
    F(x)(1-F(x))\leq \frac{1+\E |\xi|^{s}}{1+|x|^{s}},\quad \forall x\in\mathbb R.
\end{equation}

The above discussion leads to the following proof.
\begin{proof}[Proof of Theorem~\ref{thm: empirical}]
Notice that for $\theta\in S$, $\langle \theta, X_1\rangle,...,\langle \theta, X_n\rangle$ is a sample drawn from $\hat\mu_{\theta}$ and $(\hat\mu_{\theta})^{n}=\frac1n\sum\limits_{k=1}^n\delta_{\langle \theta, X_k\rangle}$. Let $F_{\theta}$ denote the distribution function of $\hat\mu_{\theta}$ and $X_{\theta}\sim \hat\mu_{\theta}$ . Applying Theorem~\ref{thm: empirical_in_1d}, we obtain
\begin{align*}
\MoveEqLeft\E W_p^p((\hat\mu_{\theta})^{n},\hat\mu_{\theta})\leq \frac{p2^{p-1}}{\sqrt{n}}\int_{-\infty}^{\infty}|x|^{p-1}\sqrt{F_{\theta}(x)(1-F_{\theta}(x))}\dd x\\
&\overset{\mathclap{\eqref{est: F(1-F)}}}{\leq}\frac{p2^{p-1}}{\sqrt{n}}\int_{-\infty}^{\infty}|x|^{p-1}\left(\frac{1+\E |X_{\theta}|^s}{1+|x|^{s}}\right)^{\frac12}\dd x\overset{\eqref{est: p moment}}{\leq}\frac{p2^{p}}{\sqrt{n}}(1+M_s(\mu))^{\frac12}\int_{0}^{\infty}\frac{|x|^{p-1}}{(1+|x|^{s})^{\frac12}}\dd x\\
&\leq\frac{p2^{p}}{\sqrt{n}}(1+M_s(\mu))\left (1+\int_{1}^{\infty}|x|^{p-1-\frac s2}\dd x    \right )=\frac{p2^{p}}{\sqrt{n}}(1+M_s(\mu))^{\frac12}\left(1+\frac{1}{p-\frac{s}{2}}\right).
\end{align*}
Tonelli's theorem gives that
$\E (SW_p^{\gamma}(\mu^n,\mu))^p\leq \frac{p2^{p}}{\sqrt{n}}(1+M_s(\mu))^{\frac12}\left(1+\frac{1}{p-\frac{s}{2}}\right),$
which by Jensen's inequality implies
\eqref{conv_rate: empirical}.
\end{proof}

We then provide the following straightforward corollary estimating the sliced Wasserstein distance between two unknown measures, whose proof only uses triangle inequality.
\begin{corollary}
If $\mu,\nu\in\mathcal P_s(X)$ for $s>2p$, $\mu^n:=\frac1n\sum\limits_{k=1}^n\delta_{X_k}$, $\nu^m:=\frac1m\sum\limits_{k=1}^m\delta_{Y_k}$ with $X_1,...,X_n$ a sample drawn from $\mu$, and $Y_1,...,Y_m$ a sample drawn from $\nu$, then
\[
\E |SW^{\gamma}_p (\mu^n,\nu^m)-SW^{\gamma}_p (\mu,\nu)|\leq C \left(n^{-\frac{1}{2p}}+m^{-\frac{1}{2p}}\right),
\]
where the constant $C$ is determined by $p,s$ and $M_s(\mu),M_s(\nu)$.
\end{corollary}
\begin{remark}\label{remark: empirical}
The above results are consistent with that in~\cite{Manole_2022} where the convergence rate of $SW$ for measures on $\mathbb R^d$ does not depend on the dimension $d$. The reason is that the projection induces the problem to the uniform estimation of Wasserstein distance in one dimension. 
\end{remark}
\subsection{Comparison to quantization in Wasserstein metric} \label{subsec: comparison_Wasserstein}
In this subsection we display some results of the convergence rate of empirical measure under Wasserstein distance for measures defined both on finite and infinite dimensional spaces~\cite{fournier2013rate,10.3150/19-BEJ1151} in comparison with our results in Section~\ref{section: empirical}.

For measures on finite dimensional spaces, Wasserstein distance suffers from curse of dimensions. To be specific, we refer to the results in~\cite{fournier2013rate}, which read
\begin{theorem}[Theorem 1 in~\cite{fournier2013rate}] Let $\mu\in\mathcal P(\mathbb R^d)$ and $p>0$. Assume $M_q(\mu)<\infty$ for some $q>p$. For $n\geq 1$, let $\mu^n:=\frac1n\sum\limits_{k=1}^n\delta_{X_k}$ with $X_1,...,X_n$ a sample drawn from $\mu$.There exists a constant $C$ depending only on $p,q,d$ such that, for all $n\geq 1$, 
\[
\E W_p^p(\mu^n,\mu)\leq CM_p^{p/q}(\mu)\left\{\begin{array}{ll}
    n^{-\frac12} + n^{-\frac{(q-p)}{q}}& \text{if\:} p>d/2 \text{\:and\:} q\neq 2p \\
    n^{-1/2}\log(1+n)+n^{-\frac{(q-p)}{q}} &  \text{if\:}p>d/2 \text{\:and\:} q\neq 2p\\
     n^{-p/d}+n^{-\frac{(q-p)}{q}}& \text{if\:}0<p<\frac d2\text{\:and\:} q\neq \frac{d}{d-p}.
\end{array}\right.
\]
\end{theorem}
It follows that if $p$ is fixed and when $d$ is large, the dominant term in the convergence rate will be $n^{-p/d}$ approaching $1$.

On the other hand, for measures on infinite dimensional spaces,~\cite{10.3150/19-BEJ1151} studied the convergence rate of Wasserstein distance between certain class of infinite dimensional measures and their empirical measures. We will state their results here. The probability measures are defined on a Hilbert space $\mathcal X=L^2=\{x\in\mathbb R^{\infty}:\sum\limits_{m=1}^{\infty}x_m^2<\infty\}$. 
\begin{theorem}[Theorem 4.1 in~\cite{10.3150/19-BEJ1151}, Polynomial Decay]
Define the distribution class
\[
\mathcal P_{poly}(q,b,M_q):=\left\{\:\mu:\: \E_{X\sim\mu}\big[ \sum\limits_{m=1}^{\infty}(m^bX_m)^2\big]^{\frac q2}\leq M_q^q\right\}.
\]
If $p, q, b$ are constants such that $1 \leq p < q$ and $b > \frac12$, then there exist
positive constants $\underline{c}_{p,q,b}$, $\bar{c}_{p,q,b}$ depending on $(p,q,b)$ such that
\[
\underline{c}_{p,q,b} M_q(\log n)^{-b}\leq \sup\limits_{\mu\in\mathcal P_{poly}(q,b,M_q)}\E W_p(\mu^n,\mu)\leq \bar{c}_{p,q,b}M_q(\log n)^{-b}.
\]
\end{theorem}
\begin{theorem}[Theorem 4.2 in~\cite{10.3150/19-BEJ1151}, Exponential Decay]
Define the distribution class
\[
\mathcal P_{exp}(q,\alpha,M_q):=\left\{\:\mu:\: \E_{X\sim\mu}\big[ \sum\limits_{m=1}^{\infty}(\alpha^{m-1}X_m)^2\big]^{\frac q2}\leq M_q^q\right\}.
\]
If $p, q, \alpha$ are constants such that $1 \leq p < q$ and $\alpha > 1$, then there exist
positive constants $\underline{c}_{p,q,\alpha}$, $\bar{c}_{p,q,\alpha}$ depending on $(p,q,\alpha)$ such that
\[
\underline{c}_{p,q,\alpha} M_q e^{-\sqrt{\log\alpha\log n}}\leq \sup\limits_{\mu\in\mathcal P_{exp}(q,\alpha,M_q)}\E W_p(\mu^n,\mu)\leq \bar{c}_{p,q,\alpha}M_qe^{-\sqrt{\log\alpha\log n}}.
\]
\end{theorem}

The convergence rate in Wasserstein distance is a finite power of $(\log n)^{-1}$ for polynomial decay and a finite power of $e^{-\sqrt{\log n}}$ for exponential decay, both of which are significantly slower than that of $n^{-\frac{1}{2p}}$ in Theorem~\ref{thm: empirical}. We conclude that the sliced Wasserstein distance indeed reduces the computational complexity.

\section*{Acknowledgements}
The author sincerely thanks Professor Dejan Slep\v{c}ev and Sangmin Park for the stimulating discussion and precious advice.

\bibliographystyle{alpha}
\bibliography{reference}
\end{document}